\documentclass[11pt]{amsart}

\usepackage[utf8]{inputenc}
\usepackage[T1]{fontenc}

\usepackage{amsmath,amssymb,amsfonts,textcomp,amsthm,xifthen,psfrag,graphicx,color}
\usepackage{ifthen}

\usepackage{enumerate}

\usepackage{fullpage}

\usepackage[utf8]{inputenc}

\usepackage[pdftex,
            pdfauthor={Thomas F\"uhrer},
            pdftitle={Superconvergence in a DPG method for an ultra-weak formulation},
            ]{hyperref}

\newtheorem{theorem}{Theorem}
\newtheorem{lemma}[theorem]{Lemma}
\newtheorem{corollary}[theorem]{Corollary}
\newtheorem{proposition}[theorem]{Proposition}
\newtheorem{remark}[theorem]{Remark}

\newcommand{\err}{\operatorname{err}}

\def\eps{\varepsilon}

\newcommand{\ip}[2]{(#1\hspace*{.5mm},#2)}
\newcommand{\dual}[2]{\langle#1\hspace*{.5mm},#2\rangle}

\newcommand{\norm}[3][]{#1\|#2#1\|_{#3}}

\newcommand{\snorm}[2]{|#1|_{#2}}

\newcommand{\diam}{\mathrm{diam}}

\def\div{{\rm div\,}}

\newcommand{\Hdivset}[1]{\boldsymbol{H}(\div;#1)}

\newcommand{\set}[2]{\big\{#1\,:\,#2\big\}}

\newcommand{\pwnabla}{\nabla_\TT}
\newcommand{\pwdiv}{\div_\TT}

\newcommand{\RT}{\ensuremath{\mathcal{RT}}}

\newcommand{\R}{\ensuremath{\mathbb{R}}}
\newcommand{\N}{\ensuremath{\mathbb{N}}}

\newcommand{\vv}{\ensuremath{\boldsymbol{v}}}
\newcommand{\ww}{\ensuremath{\boldsymbol{w}}}
\newcommand{\TT}{\ensuremath{\mathcal{T}}}
\newcommand{\MM}{\ensuremath{\mathcal{M}}}

\newcommand{\cS}{\ensuremath{\mathcal{S}}}

\newcommand{\el}{\ensuremath{T}}

\newcommand{\PP}{\ensuremath{\mathcal{P}}}
\newcommand{\OO}{\ensuremath{\mathcal{O}}}

\newcommand{\normal}{\ensuremath{{\boldsymbol{n}}}}

\newcommand{\ssigma}{{\boldsymbol\sigma}}
\newcommand{\ttau}{{\boldsymbol\tau}}
\newcommand{\llambda}{{\boldsymbol\lambda}}
\newcommand{\cchi}{{\boldsymbol\chi}}
\newcommand{\qq}{{\boldsymbol{q}}}
\newcommand{\uu}{\boldsymbol{u}}
\newcommand{\eeps}{\boldsymbol{\varepsilon}}

\newcounter{constantsnumber}
\def\setc#1{
  \ifthenelse{\equal{#1}{poinc}}{C_{\rm edge}}{ 
   \refstepcounter{constantsnumber}
   \label{const#1}C_{\theconstantsnumber}}}

\def\c#1{
  \ifthenelse{\equal{#1}{poinc}}{C_{\rm edge}}{ 
    C_{\ref{const#1}}}}

\begin{document}

\title[Superconvergent DPG methods]{Superconvergence in a DPG method for an ultra-weak formulation}
\date{\today}

\author{Thomas F\"{u}hrer}
\address{Facultad de Matem\'{a}ticas, Pontificia Universidad Cat\'{o}lica de Chile, Santiago, Chile}
\email{tofuhrer@mat.uc.cl}

\thanks{{\bf Acknowledgment.} 
This work was 
supported by FONDECYT project 11170050 and FWF project \emph{Optimal adaptivity for BEM and
FEM-BEM coupling} under grant P27005. Special thanks go to Norbert Heuer (Pontifica Universidad Cat\'olica de Chile) 
for fruitful discussions on the topic of this research.}

\begin{abstract}
  In this work we study a DPG method for an ultra-weak variational formulation of a reaction-diffusion problem.
  We improve existing a priori convergence results by sharpening an approximation result for the numerical flux.
  By duality arguments we show that higher convergence rates for the scalar
  field variable are obtained if the polynomial order of the corresponding approximation space is increased by one.
  Furthermore, we introduce a simple elementwise postprocessing of the solution and prove superconvergence.
  Numerical experiments indicate that the obtained results are valid beyond the underlying model problem.
\end{abstract}
\maketitle

\section{Introduction}
The discontinuous Petrov Galerkin method with optimal test functions (DPG method), introduced in the present form by
Demkowicz \& Gopalakrishnan in a series of papers~\cite{partI,partII} and~\cite{partIII,partIV},
has gained much attention due to its inherent stability properties.
This makes the method particularly interesting for singularly perturbed problems~\cite{BroersenStevenson_2014,DemkowiczHeuer_2013,HeuerK_2017}, where
the aim is to provide robust error bounds.
Other features of the DPG methods are, for instance, its equivalence to a minimal residual method or its ``built-in''
error estimator, for the latter see~\cite{DPGaposteriori}.

DPG methods use broken test spaces, which allow for an efficient calculation of (nearly) optimal test functions and
the feasible evaluation of the residual in the dual norm of the test space.
Different variational formulations can be used within the DPG framework. Very popular
is the ultra-weak formulation, see~\cite{DemkowiczG_11_ADM} for the analysis in the case of a Poisson model problem.
For second order elliptic problems, ultra-weak formulations are derived by recasting the PDE into a first-order system,
multiplying with test functions, and then integrating by parts elementwise.
This requires the introduction of new trace unknowns that only live on the skeleton of the mesh. These variables
represent the trace resp. normal flux of the solution on element boundaries.
Another DPG method based on a variational formulation without recasting the PDE into a first-order system is introduced
and analyzed for the Poisson model problem in~\cite{DemkowiczG_13_PDM}.

The purpose of this note is to study convergence rates for DPG methods using the ultra-weak formulation.
One particular drawback of ultra-weak formulations is that $u$ and its gradient $\ssigma = \nabla u$ 
are usually approximated simultaneously with the same order.
Therefore, we are interested if higher convergence rates for the scalar field variable $u$ are possible, either by
increasing the polynomial order of the corresponding approximation space or by defining an approximation of the scalar
field variable by a suitable postprocessing.
The latter issue has been raised and addressed for other numerical schemes too, e.g., the HDG method~\cite{HDG}.
In the context of DG methods,
the achievement of higher convergence rates by postprocessing the solution has been studied
thoroughly and is called superconvergence, see, 
e.g.~\cite{CockburnDG_superconvergence,Cockburn_GS_projection,Cockburn_GW_superconvergence}.

Let us also note that convergence rates for primal DPG methods were studied and analyzed in~\cite{BoumaGH_DPGconvRates}. 
There, duality arguments have been introduced and used to prove higher convergence rates of the field variable in weaker
norms.
We point out that by using an ultra-weak formulation, we only have weak norms at hand and thus the ideas
from~\cite{BoumaGH_DPGconvRates} can not directly be transferred to our situation.
In particular, higher convergence rates for ultra-weak formulations can only be obtained if we augment the approximation
space of the scalar field variable or by a postprocessing of the solution.

\subsection{Overview and main contributions of this work}
The remainder of this note and the main results are given as follows:
\begin{itemize}
  \item In Section~\ref{sec:dpg} we introduce the model problem together with the DPG method based on the
    ultra-weak variational formulation. We also recall some important results from the literature that are used in the
    proofs that follow.
  \item Section~\ref{sec:apriori} deals with a priori convergence results for the overall error. We improve the existing
    analysis in the sense that only minimal regularity of the solution $u$ is required to obtain optimal convergence rates.
  \item In Section~\ref{sec:an} we augment the trial space for approximations $u_h\approx u$, i.e., we 
    seek $u_h$ in a polynomial space of order $p+1$ and $\ssigma_h\approx \ssigma=\nabla u$ in a polynomial space of
    order $p$. We develop duality arguments and prove in Theorem~\ref{thm:an} that the error $u-u_h$ in the
    $L^2(\Omega)$ norm converges at a higher rate than the overall error (which also includes the error for the
    skeleton unknowns).
    The practical relevance of augmenting the trial space is that existing finite element codes for DPG methods
    can be used with minor modification.
  \item In Section~\ref{sec:postproc} we seek approximations $u_h\approx u$ and $\ssigma_h\approx \ssigma$ in polynomial
    approximation spaces of the same order $p$. Then, we introduce the approximation $\widetilde u_h$ by postprocessing
    the solution components $u_h,\ssigma_h$. To be more precise, $\widetilde u_h$ is elementwise given as the solution
    of a discrete Neumann problem.
    In Theorem~\ref{thm:postproc} we prove superconvergence, i.e., the error $u-\widetilde u_h$ in the $L^2(\Omega)$
    norm converges at a higher rate than the overall error.
  \item Finally, numerical experiments are presented in Section~\ref{sec:ex}. We consider one problem in a convex domain
    and one problem in a nonconvex domain.
    For the latter, reduced convergence rates are predicted by our analysis and also observed in the example.
\end{itemize}

\subsection{Notational convention}
In the remainder of this work we will write $A\lesssim B$ resp. $B\lesssim A$
if there exists a constant $C>0$ that is independent of the maximal
mesh-size $h$ such that $A\leq C B$ resp. $B\leq C A$. 
Moreover, $A\simeq B$ means that $A\lesssim B$ and $B\lesssim A$.
This notation will be heavily used in the proofs, whereas in all statements of
the following theorems and corollaries we explicitly point out the dependence of $C>0$
on relevant quantities.

\section{DPG with ultra-weak formulation}\label{sec:dpg}

\subsection{Model problem}\label{sec:model}
Let $\Omega\subseteq \R^2$
be a bounded simply connected Lipschitz domain with polygonal boundary $\Gamma = \partial\Omega$.
We consider the reaction diffusion problem
\begin{subequations}\label{eq:model}
\begin{alignat}{2}
  -\Delta u + u &= f &\quad&\text{in }\Omega, \\
  u &=0 &\quad&\text{on }\Gamma.
\end{alignat}
\end{subequations}
It is well known that this problem admits for all $f\in L^2(\Omega)$ a unique solution
$u\in H_0^1(\Omega)$. Moreover, it is known that for some $s\in(1/2,s_\Omega]$ it holds 
that
\begin{align*}
  u\in H^{1+s}(\Omega) \quad\text{and}\quad \norm{u}{H^{1+s}(\Omega)} \leq C \norm{f}{},
\end{align*}
where $s_\Omega\in(1/2,1]$ depends on the geometry of $\Omega$. 
For a convex domain $\Omega$ it even holds that
\begin{align*}
  u\in H^2(\Omega) \quad\text{and}\quad \norm{u}{H^2(\Omega)} \leq C \norm{f}{}.
\end{align*}
Here, and throughout $H^t(\Omega)$ denotes the Sobolev space of real order $t>0$ with norm
$\norm{\cdot}{H^t(\Omega)}$ and seminorm $\snorm\cdot{H^t(\Omega)}$.

\subsection{Notation}
With $\norm\cdot{}$ resp. $\ip\cdot\cdot$ we denote the canonical norm resp. scalar product on $L^2(\Omega)$.
Let $\TT$ denote a mesh on $\Omega$ consisting of triangles,
i.e., $\bigcup_{\el\in\TT}\overline\el = \overline\Omega$. Throughout, we assume that the mesh $\TT$ is shape-regular
and set $h := \max_{\el\in\TT} \diam(\el)$.
With $\norm\cdot{\el}$ resp. $\ip\cdot\cdot_\el$ we denote the norm resp. scalar product on $L^2(\el)$ for
$\el\in\TT$.

We use the (broken) spaces 
\begin{align*}
  H^1(\TT) &:= \{ w\in L^2(\Omega) \,:\, w|_T \in H^1(T) \,\text{for all } T\in\TT\}, \\
  \Hdivset\TT &:= \{ \qq\in (L^2(\Omega))^2 \,:\, \qq|_T \in \Hdivset{T} \,\text{for all } T\in\TT\},
\end{align*}
where $\Hdivset\el := \set{\qq\in L^2(\el)^2}{\div\qq \in L^2(\el)}$.
The symbols $\pwnabla$ resp. $\pwdiv$ denote, the $\TT$-elementwise gradient resp. divergence.
The spaces are equipped with the canonical norms
\begin{align*}
  \norm{v}{H^1(\TT)}^2&:= \norm{\pwnabla v}{}^2 + \norm{v}{}^2, \qquad
  \norm{\ttau}{\Hdivset\TT}^2 := \norm{\pwdiv\ttau}{}^2 + \norm{\ttau}{}^2.
\end{align*}
and scalar products.

The set of all boundaries of all elements is called the skeleton
$\cS := \left\{ \partial\el \mid \el\in\TT \right\}$.
By $\normal_T$ we mean the outer normal vector on $\partial T$ for $T\in\TT$.
We use the skeleton trace spaces
\begin{align*}
  H^{1/2}(\cS) &:=
  \Big\{ \widehat u \in \Pi_{\el\in\TT}H^{1/2}(\partial\el)\,:\,
         \exists w\in H^1(\Omega) \text{ such that } 
         \widehat u|_{\partial\el} = w|_{\partial\el}\; \forall \el\in\TT \Big\},\\
  H^{-1/2}(\cS) &:=
  \Big\{ \widehat\sigma \in \Pi_{\el\in\TT}H^{-1/2}(\partial\el)\,:\,
         \exists \qq\in\Hdivset\Omega \text{ such that } 
         \widehat\sigma|_{\partial\el} = (\qq\cdot\normal_{\el})|_{\partial\el}\; \forall\el\in\TT \Big\}.
\end{align*}
Here, $H^{-1/2}(\partial\el)$ is the dual space of $H^{1/2}(\partial\el)$, where duality is understood with respect to
the $L^2(\partial\el)$ scalar product $\dual\cdot\cdot_{\partial\el}$.
We will use the canonical trace operators $\gamma_{0,\cS} : H^1(\Omega) \to H^{1/2}(\cS)$ and $\gamma_{\normal,\cS} :
\Hdivset\Omega :=\set{\ttau\in L^2(\Omega)^2}{\div\ttau\in L^2(\Omega)}
\to H^{-1/2}(\cS)$, i.e., for $u\in H^1(\Omega)$, $\ssigma\in\Hdivset\Omega$, and all $T\in\TT$ it holds that
\begin{align*}
  (\gamma_{0,\cS} u)|_{\partial T} = u|_{\partial T}, \qquad 
  (\gamma_{\normal,\cS} \ssigma)|_{\partial T} = (\ssigma\cdot\normal_T)|_{\partial T}.
\end{align*}
Therefore, $H^{1/2}(\cS) = \gamma_{0,\cS}(H^1(\Omega))$ and $H^{-1/2}(\cS) = \gamma_{\normal,\cS}(\Hdivset\Omega)$.
These spaces are equipped with the minimal energy extension norms 
\begin{align*}
  \norm{\widehat u}{1/2,\cS} &:= \inf\set{\norm{u}{H^1(\Omega)}}{u\in H^1(\Omega),\,
    \gamma_{0,\cS}u = \widehat u}, \\
    \norm{\widehat \sigma}{-1/2,\cS} &:= 
    \inf\set{\norm{\ssigma}{\Hdivset\Omega}}{\ssigma\in\Hdivset\Omega,\,\gamma_{\normal,\cS}\ssigma =\widehat\sigma}.
\end{align*}
Furthermore, set $H_0^{1/2}(\cS) := \gamma_{0,\cS}(H_0^1(\Omega))$.

\subsection{Discrete spaces and approximation operators}\label{sec:approx}
Let $\PP^p(\el)$ denote the space of polynomials on $\el\in\TT$ with degree less than or equal to $p\in\N_0$.
We use the standard polynomial approximation spaces
\begin{align*}
  \PP^p(\TT) &:= \set{v\in L^2(\Omega)}{v|_\el \in\PP^p(\el) \text{ for all }\el\in\TT}, \\
  \PP_c^p(\TT) &:= \PP^p(\TT)\cap C(\overline\Omega), \\
  \RT^p(\TT) &:= \set{\ttau\in \Hdivset\Omega}{\qq|_\el\in \RT^p(\el) \text{ for all }\el\in\TT},
\end{align*}
where $\RT^p(\el)$ is the Raviart-Thomas space of order $p\in\N_0$ on $\el\in\TT$.
For functions in the skeleton spaces $H_0^{1/2}(\cS)$ and $H^{-1/2}(\cS)$ we use the approximation spaces
\begin{align*}
  \PP_{c,0}^{p+1}(\cS) &:= \gamma_{0,\cS}(\PP_c^{p+1}(\TT)\cap H_0^1(\Omega)) \subseteq H_0^{1/2}(\cS), \\
  \PP^p(\cS) &:= \gamma_{\normal,\cS}(\RT^p(\TT)) \subseteq H^{-1/2}(\cS).
\end{align*}

We need several approximation operators. For $p\in\N_0$, let
$\Pi^p : L^2(\Omega) \to \PP^p(\TT)$ denote the $L^2$-orthogonal
projection. It is known that
\begin{subequations}\label{eq:approx}
\begin{align}
  \min_{v_h\in\PP^p(\TT)} \norm{u-v_h}{} = \norm{u-\Pi^pu}{} \leq C_p h^s \snorm{u}{H^s(\Omega)}
  \quad\text{for } s\in (0,p+1].
\end{align}
Moreover, the first-order local approximation property holds:
\begin{align}\label{eq:approx:b}
  \norm{v-\Pi^p v}{} \leq\norm{v-\Pi^0v}{} \leq C_0 h \norm{\pwnabla v}{} \quad\text{for all } v\in H^1(\TT).
\end{align}
Further, let $\Pi_\nabla^{p+1} : H^{1+s}(\Omega)\to \PP_c^{p+1}(\TT)\cap H_0^1(\Omega)$
denote an interpolation operator, such that
\begin{align}
  \norm{u-\Pi_\nabla^{p+1}u}{H^1(\Omega)} \leq C_p h^s \snorm{u}{H^{1+s}(\Omega)} 
\quad\text{for } s\in (1/2,p+1].
\end{align}
Let $\Pi_\div^p : H^{s}(\Omega)^2 \to \RT^p(\TT)$ denote the Raviart-Thomas projection. It holds that
\begin{align}
  \norm{\ssigma-\Pi_\div^p \ssigma}{} \leq C_p h^s 
  \snorm{\ssigma}{H^{s}(\Omega)} \quad\text{for } s\in(1/2,p+1].
\end{align}
\end{subequations}
The constant $C_p=C_p(s)>0$ in~\eqref{eq:approx} depends on shape-regularity of the mesh $\TT$ and 
the fixed polynomial degree $p\in\N_0$ and $s$, but not on $h$. 
Different choices for $\Pi_\nabla^{p+1}$ are available in the literature.
We stress that for some operators even $p$-explicit bounds are known. We refer the interested reader
to the discussion in~\cite[Section~4]{DemkowiczG_11_ADM} and the references therein.
In our analysis we also make use of the commutativity property 
\begin{align}\label{eq:RTcomm}
  \div\Pi_\div^p = \Pi^p\div.
\end{align}

\subsection{The ideal DPG method}\label{sec:iDPG}
Define the spaces
\begin{align*}
  U &:= L^2(\Omega)\times L^2(\Omega)^2 \times H_0^{1/2}(\cS) \times H^{-1/2}(\cS), \\
  V &:= H^1(\TT) \times \Hdivset\TT,
\end{align*}
which are equipped with the canonical norms, i.e.,
\begin{align*}
  \norm{\uu}U^2 &:= \norm{u}{}^2 + \norm{\ssigma}{}^2 + \norm{\widehat u}{1/2,\cS}^2 + 
  \norm{\widehat\sigma}{-1/2,\cS}^2, \\
  \norm{\vv}V^2 &:= \norm{v}{}^2 + \norm{\pwnabla v}{}^2 + \norm{\ttau}{}^2 + \norm{\pwdiv\ttau}{}^2
\end{align*}
for $\uu = (u,\ssigma,\widehat u,\widehat\sigma)\in U$, $\vv = (v,\ttau)\in V$.

We rewrite~\eqref{eq:model} as first-order system
\begin{align*}
  \ssigma-\nabla u = 0, \qquad -\div\ssigma+u = f \qquad\text{in }\Omega.
\end{align*}
Multiplication with test functions $\ttau\in\Hdivset\TT$, $v\in H^1(\TT)$, and 
elementwise integration by parts of the two resulting equations leads to the ultra-weak formulation. 
Define the ultra-weak bilinear form $b:U\times V \to \R$ by
\begin{align*}
  b(\uu,\vv) := \ip{u}{\pwdiv\ttau+v} + \ip{\ssigma}{\ttau+\pwnabla v}
  - \dual{\widehat u}{\ttau\cdot\normal}_{\cS} - \dual{\widehat\sigma}{v}_{\cS}
\end{align*}
for all $\uu = (u,\ssigma,\widehat u,\widehat\sigma)\in U$, $\vv=(v,\ttau)\in V$, where
\begin{align*}
  \dual{\widehat u}{\ttau\cdot\normal}_{\cS} := 
  \sum_{T\in\TT} \dual{\widehat u|_{\partial T}}{\ttau\cdot\normal_T}_{\partial T}, \quad
  \dual{\widehat\sigma}{v}_{\cS} := 
  \sum_{T\in\TT} \dual{\widehat\sigma|_{\partial T}}{v|_{\partial T}}_{\partial T}.
\end{align*}
Given the exact solution $u\in H_0^1(\Omega)$ of~\eqref{eq:model}, then $\uu = (u,\ssigma,\widehat u,\widehat\sigma) :=
(u,\nabla u,\gamma_{0,\cS} u,\gamma_{\normal,\cS}\nabla u) \in U$ satisfies
\begin{align}\label{eq:uwf}
  b(\uu,\vv) = F(\vv) \quad\text{for all }\vv = (v,\ttau)\in V,
\end{align}
where $F(\vv) := \ip{f}v$.
The bilinear form $b(\cdot,\cdot)$ is uniformly bounded with bound $\norm{b}{}$ and
satisfies the $\inf$--$\sup$ conditions with constant $C_b>0$ depending only on $\Omega$,
see~\cite{DemkowiczG_11_ADM} for the analysis of the Poisson
problem and~\cite{DPGtimestepping} for the analysis with additional reaction term.
Thus, the formulation~\eqref{eq:uwf} has a unique solution for all $f\in L^2(\Omega)$.
Furthermore, define $B : U \to V'$ by $(B\uu)(\vv) := b(\uu,\vv)$ for $\uu\in U, \vv\in V$.

To define the DPG method we need the trial-to-test operator $\Theta : U\to V$, which is
defined by the relation
\begin{align}\label{eq:ttt}
  \ip{\Theta \uu}{\vv}_V = b(\uu,\vv) \quad\text{for all }\vv\in V.
\end{align}
Here, $\ip\cdot\cdot_V$ denotes the scalar product that induces $\norm\cdot{V}$.

Let $U_h\subseteq U$ be a finite-dimensional subspace. 
The ideal DPG method then reads: Find $\uu_h\in U_{h}$ such that
\begin{align}\label{eq:iDPG}
  b(\uu_h,\Theta\ww_h) = F(\Theta\ww_h) \quad\text{for all }\ww_h\in U_{h}.
\end{align}
We recall some basic results, see, e.g.~\cite{DemkowiczG_11_ADM} for the Poisson problem and~\cite{DPGtimestepping} for
the reaction-diffusion problem.

\begin{proposition}\label{prop:equivalence}
  It holds that
  \begin{align*}
    C^{-1}\norm{\uu}U^2 \leq \norm{B\uu}{V'}^2 = b(\uu,\Theta\uu) \leq C \norm{\uu}U^2 \quad\text{for all }\uu\in U,
  \end{align*}
  where $C>0$ depends only on $\Omega$.
  \qed
\end{proposition}

\begin{proposition}\label{prop:iDPG}
  Let $f\in L^2(\Omega)$ and $U_h\subseteq U$ be a finite-dimensional subspace.
  Let $\uu\in U$ denote the unique solution of~\eqref{eq:uwf} and let
  $\uu_h\in U_{h}$ denote the unique solution of~\eqref{eq:iDPG}.
  Then,
  \begin{align*}
    \norm{B(\uu-\uu_h)}{V'} = \min_{\ww_h\in U_{h}} \norm{B(\uu-\ww_h)}{V'}.
  \end{align*}
  In particular, quasi-optimality
  \begin{align*}
    \norm{\uu-\uu_h}U \leq C_\mathrm{opt} \min_{\ww_h\in U_{h}} \norm{\uu-\ww_h}U
  \end{align*}
  holds, where $C_\mathrm{opt}>0$ depends only on $\Omega$.
  \qed
\end{proposition}

\subsection{Practical DPG method}\label{sec:practicalDPG}
In this section we will consider the \emph{practical DPG method}. The idea is to replace the test space $V$ by a finite
dimensional space $V_r\subseteq V$.
We assume that there exists a so-called \emph{Fortin} operator $\Pi:V\to V_r$ with the following properties:
There exists $C_\Pi>0$, independent of $\TT$, such that
\begin{align}\label{eq:fortin}
  b(\uu_h,\vv-\Pi\vv) &= 0 \quad\text{and}\quad
  \norm{\Pi\vv}V \leq C_\Pi \norm{\vv}V \qquad\text{for all }\uu_h\in U_h, \vv\in V.
\end{align}
Under these assumptions it is known that the following practical DPG method admits a unique solution. 
Find $\uu_h\in U_h$ such that
\begin{align}\label{eq:practicalDPG}
  b(\uu_h,\Theta_r\ww_h) = F(\Theta_r\ww_h) \quad\text{for all }\ww_h\in U_h,
\end{align}
where the discrete trial-to-test operator $\Theta_r\ww_h \in V_r$ is defined through
\begin{align*}
  \ip{\Theta_r\ww_h}{\vv_r}_V = b(\ww_h,\vv_r) \quad\text{for all }\vv_r\in V_r.
\end{align*}

For our analysis below it is more convenient to use the mixed formulation of the practical DPG method:
Find $(\uu_h,\eeps_r)\in U_h\times V_r$ such that
\begin{subequations}\label{eq:mixedDPG}
  \begin{alignat}{4}
    &\ip{\eeps_r}{\vv_r}_V + b(\uu_h,\vv_r) &\,=\,& F(\vv_r) &\qquad&\text{for all } \vv_r\in V_r,
    \label{eq:mixedDPG:a}\\
    &b(\ww_h,\eeps_r) &\,=\,& 0 &\qquad&\text{for all } \ww_h \in U_h. \label{eq:mixedDPG:b}
  \end{alignat}
\end{subequations}
This variational formulation is actually equivalent to~\eqref{eq:practicalDPG},
see~\cite[Theorem~2.4]{BoumaGH_DPGconvRates} resp.~\cite{notesJayG} for a proof.
Define the bilinear form $a: (U\times V)\times (U\times V) \to \R$ by
\begin{align*}
  a( (\uu,\eeps),(\ww,\vv) ) := \ip{\eeps}\vv_V + b(\uu,\vv) - b(\ww,\eeps)
\end{align*}
for all $(\uu,\eeps),(\ww,\vv)\in U\times V$. 
Note that $a(\cdot,\cdot)$ is uniformly continuous with boundedness constant $\norm{a}{}$ depending only on
$\norm{b}{}$.
Let $\uu\in U$ denote the solution of~\eqref{eq:uwf} and set $\eeps :=0$. Then,
\begin{align*}
  a( (\uu,\eeps), (\ww,\vv) ) = b(\uu,\vv) = F(\vv) \quad\text{for all } (\ww,\vv) \in U\times V.
\end{align*}
Particularly, with $(\uu_h,\eeps_r)$ being the solution of~\eqref{eq:mixedDPG},
we emphasize that Galerkin orthogonality holds:
\begin{align*}
  a( (\uu-\uu_h,\eeps-\eeps_r), (\ww_h,\vv_r) ) = 0 \quad\text{for all } (\ww_h,\vv_r)\in U_h\times V_r.
\end{align*}

A similar result as in Proposition~\ref{prop:iDPG} can be derived for the practical DPG method,
see~\cite[Theorem~2.1]{practicalDPG} for a proof of the
quasi-optimality, and also~\cite{notesJayG}.
\begin{proposition}\label{prop:practicalDPG}
  Let $f\in L^2(\Omega)$ and $U_h\subseteq U$ be a finite-dimensional subspace.
  Let $\uu\in U$ denote the unique solution of~\eqref{eq:uwf} and suppose~\eqref{eq:fortin} holds.
  Let $\uu_h\in U_{h}$ denote the unique solution of~\eqref{eq:practicalDPG}.
  Then,
  \begin{align*}
    \norm{B(\uu-\uu_h)}{V_r'} = \min_{\ww_h\in U_{h}} \norm{B(\uu-\ww_h)}{V_r'}.
  \end{align*}
  In particular, quasi-optimality
  \begin{align*}
    \norm{\uu-\uu_h}U \leq C_\mathrm{opt} \min_{\ww_h\in U_{h}} \norm{\uu-\ww_h}U
  \end{align*}
  holds, where $C_\mathrm{opt}>0$ depends only on $\Omega$ and $C_\Pi$.
  \qed
\end{proposition}

For the remainder of this work we will use the following approximation spaces.
We define for $p\in\N_0$
\begin{align*}
  U_{hp} := \PP^p(\TT) \times  \PP^p(\TT)^2 \times
  \PP_{c,0}^{p+1}(\cS) \times \PP^p(\cS),
\end{align*}
as well as the augmented space
\begin{align*}
  U_{hp}^+ := \PP^{p+1}(\TT) \times \PP^p(\TT)^2 \times
  \PP_{c,0}^{p+1}(\cS) \times \PP^p(\cS).
\end{align*}

\begin{remark}\label{rem:testSpace}
  In this work we use polynomial spaces $U_h \in \{U_{hp},U_{hp}^+\}$
  and polynomial spaces $V_r$ defined on triangular meshes $\TT$.
  Specifically, 
  \begin{align}\label{eq:assVr}
    V_r := \PP^{r_1}(\TT) \times \PP^{r_2}(\TT)^2, \qquad
    \text{where at least } r_1,r_2\geq p+2. 
  \end{align}
  Hence, we have $\PP_{c,0}^1(\TT)\times \RT^{0}(\TT) \subseteq 
  \PP_{c,0}^1(\TT)\times \RT^{p+1}(\TT)\subseteq V_r$. 
  This property
  is used in the proofs of Theorem~\ref{thm:postproc} resp. Theorem~\ref{thm:an} below.
  Fortin operators for that particular case exist, see~\cite{practicalDPG} for their definition and analysis.
  Let us note that in~\cite{practicalDPG} it was proven that $r_1=p+2$ and $r_2=p+2$ is enough to guarantee the
  existence of Fortin operators for the ultra-weak formulation of the Poisson problem in 2D using the trial space 
  $U_{hp}$. Therefore, it is also guaranteed for the augmented trial space $U_{hp}^+$ if $r_1=r_2=p+3$, since
  $U_{hp}^+\subset U_{h(p+1)}$.
  Nevertheless, for the numerical experiments from Section~\ref{sec:ex} we choose throughout $r_1=r_2=p+2$, see
  the discussion at the beginning of Section~\ref{sec:ex}.
  \qed
\end{remark}

\section{A priori convergence}\label{sec:apriori}
In this short section we refine the a priori convergence analysis.
The standard analysis so far, see, e.g.,~\cite{DemkowiczG_11_ADM}, estimates
\begin{align*}
  \min_{\widehat\tau_h\in \PP^p(\cS)}\norm{\widehat \sigma- \widehat\tau_h}{-1/2,\cS}
  \leq \norm{\ssigma-\Pi_\div^p\ssigma}{\Hdivset\Omega},
\end{align*}
where $\widehat\sigma = \gamma_{\normal,\cS} \ssigma$
and, then, applies approximation results of the Raviart-Thomas projector 
in the $\Hdivset\Omega$ norm. 
To be more precise, standard estimates give
\begin{align*}
  \norm{\ssigma-\Pi_\div^p\ssigma}{\Hdivset\Omega} 
  &\leq \norm{\ssigma-\Pi_\div^p\ssigma}{} + \norm{\div(\ssigma-\Pi_\div^p\ssigma)}{} \\
  &\lesssim h^s \norm{\ssigma}{H^s(\Omega)} + \norm{(1-\Pi^p)\div\ssigma}{}.
\end{align*}
To obtain powers of $h$ for the second term, we have to assume some regularity of $\div\ssigma = \Delta u = u-f$,
which means regularity assumptions on $f$.
This is too rough, but can be sharpened using the following result.
\begin{theorem}\label{thm:Hm12est}
  Let $p\in\N_0$.
  Let $\ssigma\in H^{s}(\Omega)^2\cap\Hdivset\Omega$ for some $s\in(1/2,p+1]$.
  Then,
  \begin{align*}
     \norm{\gamma_{\normal,\cS}(1-\Pi_\div^p)\ssigma}{-1/2,\cS}
     \leq C \cdot
     \begin{cases}
        h^s \norm{\ssigma}{H^s(\Omega)}+ h\norm{\div\ssigma}{} &\text{if } s\in(1/2,1), \\
        h^s \norm{\ssigma}{H^s(\Omega)} & \text{if } s\in[1,p+1].
     \end{cases}
  \end{align*}
  The constant $C>0$ depends only on the constants from~\eqref{eq:approx}.
\end{theorem}
\begin{proof}
  In~\cite[Theorem~2.3]{breakSpace} the identity
  \begin{align}\label{eq:identity}
    \norm{\widehat\sigma}{-1/2,\cS} = \sup_{0\neq v\in H^1(\TT)} 
    \frac{\dual{\widehat\sigma}{v}_\cS}{\norm{v}{H^1(\TT)}} \quad\text{for all }
    \widehat\sigma\in H^{-1/2}(\cS)
  \end{align}
  was proven.
  Since $\ssigma\in H^s(\Omega)^2$ for some $s>1/2$, $\Pi_\div^p\ssigma$ is well-defined and
  $\gamma_{\normal,\cS}\Pi_\div^p\ssigma\in \PP^p(\cS)$. With piecewise integration by
  parts, the commutativity property~\eqref{eq:RTcomm}, and the projection property of $\Pi^p$, we obtain
  \begin{align*}
    \dual{\gamma_{\normal,\cS}(\ssigma-\Pi_\div^p\ssigma)}v_\cS &=
    \ip{\ssigma-\Pi_\div^p\ssigma}{\pwnabla v} + \ip{\div(\ssigma-\Pi_\div^p\ssigma)}{v} \\
    &=\ip{\ssigma-\Pi_\div^p\ssigma}{\pwnabla v} + \ip{(1-\Pi^p)\div\ssigma}{v} \\
    &=\ip{\ssigma-\Pi_\div^p\ssigma}{\pwnabla v} + \ip{(1-\Pi^p)\div\ssigma}{(1-\Pi^p)v}.
  \end{align*}
  Then, the approximation properties~\eqref{eq:approx} show
  \begin{align}\label{eq:dualest}
    \dual{\gamma_{\normal,\cS}(\ssigma-\Pi_\div^p\ssigma)}v_\cS 
    \lesssim h^s\norm{\ssigma}{H^s(\Omega)}\norm{\pwnabla v}{} 
    + h \norm{(1-\Pi^p)\div\ssigma}{}\norm{\pwnabla v}{}.
  \end{align}
  For $s\in(1/2,1)$ we employ boundedness of $(1-\Pi^p)$ to see
  \begin{align*}
    h\norm{(1-\Pi^p)\div\ssigma}{} \leq h\norm{\div\ssigma}{},
  \end{align*}
  and for $s\in[1,p+1]$ we again use~\eqref{eq:approx} to get
  \begin{align*}
    h\norm{(1-\Pi^p)\div\ssigma}{} \lesssim h h^{s-1}\norm{\div\ssigma}{H^{s-1}(\Omega)}
    \lesssim h^s\norm{\ssigma}{H^s(\Omega)}.
  \end{align*}
  Using the last two estimates in~\eqref{eq:dualest}, dividing by $\norm{v}{H^1(\TT)}$ and taking the supremum over
  $H^1(\TT)\setminus\{0\}$, the assertion follows from~\eqref{eq:identity}.
\end{proof}
    
\begin{corollary}\label{cor:apriori}
  Let $f\in L^2(\Omega)$ and let $\uu\in U$ denote the unique solution of~\eqref{eq:uwf}.
  Let $p\in\N_0$. 
  Suppose $u\in H^{1+s}(\Omega)$ for some $s\in(1/2,p+1]$.
  Suppose there exist Fortin operators~\eqref{eq:fortin} for the spaces $U_{hp}$ and $U_{hp}^+$.
  Let $\uu_h\in U_h\in\{U_{hp},U_{hp}^+\}$ be the unique solution of~\eqref{eq:practicalDPG}. 
  Then,
  \begin{align*}
    \norm{\uu-\uu_h}U \leq C\cdot \begin{cases}
       h^{s}\norm{u}{H^{1+s}(\Omega)} + h \norm{f}{} & \text{if } s\in(1/2,1), \\
       h^{s}\norm{u}{H^{1+s}(\Omega)} &\text{if } s\in[1,p+1].
    \end{cases}
  \end{align*}
  The constant $C>0$ depends only on $\Omega$, the constants from~\eqref{eq:approx}, and $C_\Pi$.
  \qed
\end{corollary}
\begin{proof}
  Set $\ww_h = (\Pi^p u,\Pi^p\nabla u,\gamma_{0,\cS}\Pi_\nabla^{p+1}u,\gamma_{\normal,\cS}\Pi_\div^p\nabla u) \in U_h$. 
  Then, with the approximation properties~\eqref{eq:approx} we get
  \begin{align*}
    \norm{\uu-\ww_h}U &\lesssim h^s\norm{u}{H^s(\Omega)} + h^s\norm{\nabla u}{H^s(\Omega)}
    + h^s\norm{u}{H^{1+s}(\Omega)} + \norm{\gamma_{\normal,\cS}(1-\Pi_\div^p)\nabla u}{-1/2,\cS} \\
    &\lesssim h^s\norm{u}{H^{1+s}(\Omega)} 
    + \norm{\gamma_{\normal,\cS}(1-\Pi_\div^p)\nabla u}{-1/2,\cS}.
  \end{align*}
  The last term is tackled with Theorem~\eqref{thm:Hm12est}. For $s\geq 1$ the application is straightforward. For
  $s\in(1/2,1)$ note that $\div(\nabla u) = u-f \in L^2(\Omega)$ since $u$ is the solution of~\eqref{eq:model}. Thus,
  \begin{align*}
    h\norm{\div\nabla u}{} \lesssim h\norm{u}{} + h\norm{f}{}.
  \end{align*}
  The estimates above together with the quasi-optimality from Proposition~\ref{prop:practicalDPG} finish the proof.
\end{proof}

\begin{remark}
  If we have regularity $u\in H^2(\Omega)$, then for the lowest-order case, i.e., $p=0$, we get
  \begin{align*}
    \norm{\uu-\uu_h}U \lesssim h \norm{u}{H^2(\Omega)}.
  \end{align*}
  In particular, this resembles results for standard FEM schemes.

  Moreover, the result from Corollary~\ref{cor:apriori} can be derived in the same fashion 
  for other DPG formulations that use the trace
  space $H^{-1/2}(\cS)$, e.g., the primal DPG formulation defined and analyzed in~\cite{DemkowiczG_13_PDM}.
  \qed
\end{remark}

\section{Augmented trial space}\label{sec:an}
In this section we consider discrete solutions $\uu_h = (u_h,\ssigma_h,\widehat u_h,\widehat \sigma_h)$
from the augmented space $U_{hp}^+$ and prove with duality
arguments that $\norm{u-u_h}{}$ converges at a higher rate than the total error $\norm{\uu-\uu_h}U$.
In Section~\ref{sec:an:aux} we collect some auxiliary results and in Section~\ref{sec:an:main} we state the main result
of this section along with its proof.

\subsection{Auxiliary results}\label{sec:an:aux}
For the proof of Theorem~\ref{thm:an} we need some auxiliary results that we develop in the following.

\begin{lemma}\label{lem:mixed}
  Let $g\in L^2(\Omega)$. There exists a unique solution $\vv = (v,\ttau) \in H_0^1(\Omega)\times \Hdivset\Omega$ of the
  first-order system
  \begin{subequations}\label{eq:mixed}
  \begin{alignat}{2}
    \div\ttau + v &= g &\quad&\text{in }\Omega, \label{eq:mixed:a}\\
    \ttau + \nabla v &= 0 &\quad&\text{in }\Omega. \label{eq:mixed:b}
  \end{alignat}
  \end{subequations}
  In particular, $\norm{v}{H^1(\Omega)}+\norm{\ttau}{\Hdivset\Omega}\lesssim \norm{g}{}$.

  There exists a regularity shift $s'\in(1/2,1]$ such that
  \begin{align}\label{eq:mixed:regshift}
    v\in H^{1+s'}(\Omega) \quad\text{and}\quad \norm{v}{H^{1+s'}(\Omega)} \lesssim \norm{g}{}.
  \end{align}
  Then, for $p\in\N_0$,
  \begin{align*}
    \norm{v-\Pi_\nabla^1 v}{H^1(\Omega)} + 
    \norm{\ttau-\Pi_\div^p\ttau}{} \lesssim h^{s'} \norm{g}{}.
  \end{align*}
  If $\Omega$ is convex then $s'=1$.
\end{lemma}
\begin{proof}
  Unique solvability and stability of the first-order system is known, see, e.g.,~\cite{DemkowiczG_11_ADM} for the analysis
  of a similar problem in the context of DPG methods.

  We note that $v\in H_0^1(\Omega)$ satisfies the PDE
  \begin{align*}
    -\Delta v + v = g.
  \end{align*}
  The existence of $s'$ and~\eqref{eq:mixed:regshift} follow directly from the discussion in Section~\ref{sec:model}.
  The case of convex domains was also discussed there.
  By~\eqref{eq:mixed:b} we then have $\ttau\in H^{s'}(\Omega)^2$.
  Standard approximation results~\eqref{eq:approx} finish the proof.
\end{proof}

\begin{lemma}\label{lem:w}
  Let $g\in L^2(\Omega)$.
  There exists $\vv\in V$ and $\ww\in U$ such that 
  \begin{align*}
    \ip{u}g = b(\uu,\vv) = b(\uu,\Theta\ww) \quad\text{for all }\uu=(u,\ssigma,\widehat u,\widehat\sigma)\in U.
  \end{align*}
  In particular,
  \begin{align*}
    \ww = (g,0,\gamma_{0,\cS}v,\gamma_{\normal,\cS}\ttau),
  \end{align*}
  where $\vv=(v,\ttau)\in V$ is the unique solution of~\eqref{eq:mixed}.
\end{lemma}
\begin{proof}
  Let $\vv=(v,\ttau)\in H_0^1(\Omega)\times \Hdivset\Omega$ be the solution of~\eqref{eq:mixed}.
  Then,
  \begin{align*}
    \ip{u}g = \ip{u}{\div\ttau+v} = b(\uu,\vv).
  \end{align*}
  Since $\Theta : U \to V$ is an isomorphism, we find a unique $\ww\in U$ with $\Theta\ww = \vv$.

  We can also calculate the function $\ww$ explicitly.
  Note that by~\eqref{eq:ttt} we have to find $\ww\in U$ with
  \begin{align*}
    \ip{\vv}{(\mu,\llambda)}_V = b(\ww,(\mu,\llambda)) \quad\text{for all }(\mu,\llambda)\in V.
  \end{align*}
  Recall that $v\in H_0^1(\Omega)$ satisfies
  \begin{align*}
    -\Delta v + v = g \in L^2(\Omega).
  \end{align*}
  Multiplying this equation with $\mu\in H^1(\TT)$ and integrating by parts (elementwise) we get
  \begin{align*}
    \ip{\nabla v}{\nabla \mu}_T + \ip{v}\mu_T = \ip{g}\mu_T + \dual{\nabla v\cdot\normal_T}{\mu}_{\partial T}.
  \end{align*}
  Using~\eqref{eq:mixed:b} and summing over all elements, we obtain
  \begin{align*}
    \ip{\nabla v}{\pwnabla\mu} + \ip{v}\mu = \ip{g}\mu - \dual{\gamma_{\normal,\cS}\ttau}{\mu}_{\cS}.
  \end{align*}
  Relation~\eqref{eq:mixed:a} gives
  \begin{align*}
    \ip{\div\ttau}{\pwdiv\llambda} = \ip{g}{\pwdiv\llambda} - \ip{v}{\pwdiv\llambda}
  \end{align*}
  and~\eqref{eq:mixed:b} with integration by parts gives
  \begin{align*}
    \ip{\ttau}{\llambda} = -\ip{\nabla v}{\llambda} = \ip{v}{\pwdiv\llambda}
    -\dual{\gamma_{0,\cS}v}{\llambda\cdot\normal}_{\cS}.
  \end{align*}
  Combining the last three relations we see that
  \begin{align*}
    \ip{\vv}{(\mu,\llambda)}_V = \ip{g}{\pwdiv\llambda+\mu} - \dual{\gamma_{0,\cS}v}{\llambda\cdot\normal}_\cS
    -\dual{\gamma_{\normal,\cS}\ttau}{\mu}_\cS.
  \end{align*}
  Recall that for $\ww = (w,\cchi,\widehat w,\widehat\chi)\in U$,
  \begin{align*}
    b(\ww,(\mu,\llambda)) = \ip{w}{\pwdiv\llambda+\mu} + \ip{\cchi}{\llambda+\pwnabla \mu}
    - \dual{\widehat w}{\llambda\cdot\normal}_{\cS} - \dual{\widehat\chi}{\mu}_{\cS}.
  \end{align*}
  Comparing the last two identities, we find the explicit representation 
  \begin{align*}
    \ww = (g,0,\gamma_{0,\cS}v,\gamma_{\normal,\cS}\ttau),
  \end{align*}
  which finishes the proof.
\end{proof}

\subsection{Higher convergence rate with augmented trial space}\label{sec:an:main}
With the auxiliary results from Section~\ref{sec:an:aux} we prove the following result.
\begin{theorem}\label{thm:an}
  Let $f\in L^2(\Omega)$ and let $\uu = (u,\ssigma,\widehat u,\widehat \sigma)\in U$ denote the solution
  of~\eqref{eq:uwf}. 
  Let $p\in\N_0$ and suppose~\eqref{eq:fortin} holds with $U_h := U_{hp}^+$. 
  Let $\uu_h = (u_h,\ssigma_h,\widehat u_h,\widehat\sigma_h) \in U_h$ denote the solution of~\eqref{eq:practicalDPG}. 
  Suppose $u\in H^{1+s}(\Omega)$ for some $s\in (1/2,p+1]$.
  Moreover, set $g:=\Pi^{p+1}u-u_h\in \PP^{p+1}(\TT)$ and let $s'\in(1/2,1]$ denote the regularity
  shift~\eqref{eq:mixed:regshift} of the solution $(v,\ttau)$ of~\eqref{eq:mixed}. Then,
  \begin{align*}
    \norm{u-u_h}{} \leq C h^{s'} \cdot
    \begin{cases}
      h^s\norm{u}{H^{1+s}(\Omega)} + h \norm{f}{} &\text{if } s\in(1/2,1), \\
      h^s\norm{u}{H^{1+s}(\Omega)} &\text{if } s\in[1,p+1].
    \end{cases}
  \end{align*}
  The constant $C>0$ depends only on $\Omega$, the constants from~\eqref{eq:approx}, and $C_\Pi$.
\end{theorem}

\begin{proof}
  We consider the splitting
  \begin{align*}
    \norm{u-u_h}{}^2 = \norm{(1-\Pi^{p+1})u}{}^2 + \norm{\Pi^{p+1}u-u_h}{}^2.
  \end{align*}
  The first term is estimated with the approximation properties~\eqref{eq:approx}, 
  \begin{align*}
    \norm{(1-\Pi^{p+1})u}{} \lesssim h^{1+s}\norm{u}{H^{1+s}(\Omega)}.
  \end{align*}
  To estimate the second term set $g:=\Pi^{p+1}u-u_h\in\PP^{p+1}(\Omega)$ and let $\vv=(v,\ttau)\in H_0^1(\Omega) \times
  \Hdivset\Omega\subseteq V$ denote the solution of the first-order system~\eqref{eq:mixed}. It holds that
  \begin{align*}
    \norm{g}{}^2 &= \ip{\Pi^{p+1}u-u_h}{g} = \ip{u-u_h}g = b(\uu-\uu_h,\vv).
  \end{align*}
  Using $\ww\in U$ with $\Theta\ww=\vv$
  from Lemma~\ref{lem:w} and the definition~\eqref{eq:ttt} of the trial-to-test operator, we infer
  $\ip{\vv}{(\mu,\llambda)}_V - b(\ww,(\mu,\llambda)) = 0$ for $(\mu,\llambda)\in V$.
  Then, with the bilinear form $a(\cdot,\cdot)$, we get
  \begin{align*}
    b(\uu-\uu_h,\vv) = b(\uu-\uu_h,\vv) + \ip{\eeps-\eeps_r}{\vv}_V - b(\ww,\eeps-\eeps_r) 
    = a( (\uu-\uu_h,\eeps-\eeps_r), (\ww,\vv)).
  \end{align*}
  Recall that $\eeps=0$ and that we have Galerkin orthogonality, hence,
  together with boundedness of $a(\cdot,\cdot)$ this gives for $(\ww_h,\vv_r)\in U_h\times V_r$,
  \begin{align*}
    \norm{g}{}^2 &= a( (\uu-\uu_h,\eeps-\eeps_r), (\ww-\ww_h,\vv-\vv_r)) \\
    &\lesssim \big( \norm{\uu-\uu_h}U + \norm{\eeps_r}V \big) \big(\norm{\ww-\ww_h}U + \norm{\vv-\vv_r}V\big)
    \lesssim \norm{\uu-\uu_h}U \big(\norm{\ww-\ww_h}U + \norm{\vv-\vv_r}V\big)
  \end{align*}
  We also used that the discrete error function $\eeps_r$ satisfies the bound $\norm{\eeps_r}V\lesssim
  \norm{\uu-\uu_h}U$. This follows from the existence of a Fortin operator, 
    which allows the use of~\cite[Theorem~2.1]{DPGaposteriori}.
  We choose $(\ww_h,\vv_r)$ as
  \begin{align*}
    \ww_h &:= (g,0,\gamma_{0,\cS}\Pi_\nabla^{1}v,\gamma_{\normal,\cS}\Pi_\div^0\ttau)\in U_h = U_{hp}^+, \\
    \vv_r &:= (\Pi_\nabla^{1}v,\Pi_\div^{p+1}\ttau)\in V_r.
  \end{align*}
  Observe that $\ww-\ww_h = (0,0,\gamma_{0,\cS}(1-\Pi_\nabla^1)v,\gamma_{\normal,\cS}(1-\Pi_\div^0)\ttau)$, thus,
  \begin{align*}
    \norm{\ww-\ww_h}U &\leq \norm{\gamma_{0,\cS}(1-\Pi_\nabla^1)v}{1/2,\cS} 
    + \norm{\gamma_{\normal,\cS}(1-\Pi_\div^0)\ttau}{-1/2,\cS}
    \\&\leq 
    \norm{(1-\Pi_\nabla^1)v}{H^1(\Omega)} + \norm{\gamma_{\normal,\cS}(1-\Pi_\div^0)\ttau}{-1/2,\cS}.
  \end{align*}
  We apply Theorem~\ref{thm:Hm12est} in conjunction with Lemma~\ref{lem:mixed} and~\eqref{eq:mixed:a} to see
  \begin{align*}
    \norm{(1-\Pi_\nabla^1)v}{H^1(\Omega)} + \norm{\gamma_{\normal,\cS}(1-\Pi_\div^0)\ttau)}{-1/2,\cS}
    &\lesssim h^{s'} \norm{g}{} + h\norm{\div\ttau}{} \\
    &\lesssim h^{s'} \norm{g}{} + h\norm{g}{} + h\norm{v}{} \lesssim  h^{s'}\norm{g}{}.
  \end{align*}
  It remains to estimate $\norm{\vv-\vv_r}V$. Again we use Lemma~\ref{lem:mixed} to get
  \begin{align*}
    \norm{\vv-\vv_r}V &\leq \norm{(1\!-\!\Pi_\nabla^{1})v}{H^1(\Omega)} 
    + \norm{(1\!-\!\Pi_\div^{p+1})\ttau}{} + \norm{\div(1\!-\!\Pi_\div^{p+1})\ttau}{}
    \lesssim h^{s'}\norm{g}{} + \norm{\div(1\!-\!\Pi_\div^{p+1})\ttau}{}.
  \end{align*}
  For the last term on the right-hand side observe that, since $g\in \PP^{p+1}(\TT)$ and~\eqref{eq:assVr},
  \begin{align*}
    \div(1-\Pi_\div^{p+1})\ttau = (1-\Pi^{p+1})\div\ttau = (1-\Pi^{p+1})(g-v) = -(1-\Pi^{p+1})v,
  \end{align*}
  Thus,
  \begin{align*}
    \norm{\div(1-\Pi_\div^{p+1})\ttau}{} \lesssim \norm{(1-\Pi^{p+1})v}{} \leq \norm{(1-\Pi^0)v}{} \lesssim
    h\norm{g}{}.
  \end{align*}
  Putting all estimates together we obtain
  \begin{align*}
    \norm{g}{}^2 \lesssim \norm{\uu-\uu_h}U \big(\norm{\ww-\ww_h}U + \norm{\vv-\vv_r}V\big) 
    \lesssim \norm{\uu-\uu_h}U h^{s'} \norm{g}{}.
  \end{align*}
  Dividing by $\norm{g}{}$ we further infer for the total error
  \begin{align*}
    \norm{u-u_h}{} \leq \norm{u-\Pi^{p+1}u}{} + \norm{g}{} 
    \lesssim h^{1+s}\norm{u}{H^{1+s}(\Omega)} + h^{s'} \norm{\uu-\uu_h}U.
  \end{align*}
  Finally, applying the a priori convergence result from Corollary~\ref{cor:apriori} to estimate the global error
  $\norm{\uu-\uu_h}U$ we finish the proof.
\end{proof}

The following result is immediate.
\begin{corollary}\label{cor:an}
  Besides the assumptions of Theorem~\ref{thm:an}, suppose additionally that $\Omega$ is
  convex and $u\in H^{p+2}(\Omega)$.
  Then,
  \begin{align*}
    \norm{u-u_h}{} \leq C h^{p+2} \norm{u}{H^{p+2}(\Omega)}.
  \end{align*}
  The constant $C>0$ depends only on $\Omega$, the constants from~\eqref{eq:approx}, and $C_\Pi$.
  \qed
\end{corollary}

\section{Postprocessing of solution and superconvergence}\label{sec:postproc}
In this section we introduce a postprocessing of the solution
$\uu_h = (u_h,\ssigma_h,\widehat u_h,\widehat\sigma_h)\in U_{h}:=U_{hp}$ of~\eqref{eq:practicalDPG}.
We define $\widetilde\uu_h = (\widetilde u_h,\ssigma_h,\widehat u_h,\widehat\sigma_h)\in U_{hp}^+$, where $\widetilde
u_h\in \PP^{p+1}(\TT)$ is determined on each element $\el\in\TT$ through
\begin{subequations}\label{eq:postproc}
\begin{align}\label{eq:postproc:neumann}
  \ip{\nabla\widetilde u_h}{\nabla v}_\el &= \ip{\ssigma_h}{\nabla v}_\el \quad\text{for all } v\in\PP^{p+1}(\el), \\
  \ip{\widetilde u_h}1_\el &= \ip{u_h}1_\el, \label{eq:postproc:im}
\end{align}
\end{subequations}
Note that these relations uniquely define $\widetilde u_h\in \PP^{p+1}(\TT)$.
Observe that~\eqref{eq:postproc:neumann} is the (discretized) weak formulation of the Neumann problem
\begin{align*}
  \Delta \widetilde u &= \div\ssigma_h \quad\text{in }\el, \\
  \frac{\partial \widetilde u}{\partial \normal_\el} &= \ssigma_h\cdot\normal_\el|_{\partial \el}.
\end{align*}

Let us note that the postprocessing technique~\eqref{eq:postproc} goes back to~\cite{stenbergPostPr}.
The following result is found in a similar form in~\cite{stenbergPostPr}, resp.~\cite[Theorem~4.1]{HDGeigen}.
Therefore, the proof is omitted and left to the reader.

\begin{lemma}\label{lem:postproc}
  Let $f\in L^2(\Omega)$ and 
  let $\uu = (u,\ssigma,\widehat u,\widehat\sigma)\in U$ be the solution of~\eqref{eq:uwf}.
  For $p\in\N_0$,
  let $\uu_h = (u_h,\ssigma_h,\widehat u_h,\widehat\sigma_h) \in U_{hp}$.
  Define $\widetilde u_h \in\PP^{p+1}(\TT)$ as in~\eqref{eq:postproc}.
  Suppose $u\in H^{1+s}(\Omega)$ for some $s\in (1/2,p+1]$.
  Then,
  \begin{align*}
    \norm{u-\widetilde u_h}{} \leq C h \big(\norm{\uu-\uu_h}{U} + h^s\norm{u}{H^{1+s}(\Omega)}\big) + \norm{\Pi^0 (u-u_h)}{}.
  \end{align*}
  The constant $C>0$ depends only on the constants from~\eqref{eq:approx}.\qed
\end{lemma}

\begin{theorem}\label{thm:postproc}
  Let $f\in L^2(\Omega)$ and let $\uu = (u,\ssigma,\widehat u,\widehat \sigma)\in U$ denote the solution
  of~\eqref{eq:uwf}. 
  Let $p\in\N_0$ and suppose~\eqref{eq:fortin} holds with $U_h := U_{hp}$. 
  Let $\uu_h = (u_h,\ssigma_h,\widehat u_h,\widehat\sigma_h) \in U_h$ denote the solution of~\eqref{eq:practicalDPG}. 
  Suppose $u\in H^{1+s}(\Omega)$ for some $s\in (1/2,p+1]$.
  Define $\widetilde u_h \in\PP^{p+1}(\TT)$ as in~\eqref{eq:postproc}.
  Moreover, set $g:=\Pi^{0}(u-u_h)\in \PP^{0}(\TT)$ and let $s'\in(1/2,1]$ denote the regularity
  shift~\eqref{eq:mixed:regshift} of the 
  solution $(v,\ttau)$ of~\eqref{eq:mixed}. Then,
  \begin{align}\label{eq:postproc:convRate}
    \norm{u-\widetilde u_h}{} \leq C h^{s'} 
    \begin{cases}
      h^s\norm{u}{H^{1+s}(\Omega)} + h \norm{f}{} &\text{if } s\in(1/2,1), \\
      h^s\norm{u}{H^{1+s}(\Omega)} &\text{if } s = [1,p+1].
    \end{cases}
  \end{align}
  The constant $C>0$ depends only on $\Omega$, the constants from~\eqref{eq:approx}, and $C_\Pi$.

  In particular, if $\Omega$ is convex and $u\in H^{p+2}(\Omega)$, it holds that
  \begin{align}\label{eq:postproc:convex}
    \norm{u-\widetilde u_h}{} \leq C h^{p+2} \norm{u}{H^{p+2}(\Omega)}.
  \end{align}
\end{theorem}

\begin{proof}
  Set $g:=\Pi^0(u-u_h)$.
  Lemma~\ref{lem:postproc} shows
  \begin{align*}
    \norm{u-\widetilde u_h}{} \lesssim h \norm{\uu-\uu_h}U + h^{1+s}\norm{u}{H^{1+s}(\Omega)} + \norm{g}{}.
  \end{align*}
  The term $\norm{\uu-\uu_h}U$ can be estimated with Corollary~\ref{cor:apriori}. It thus remains to take care of
  the term $\norm{g}{}$ only.
  Let $\vv=(v,\ttau)$ be the solution of~\eqref{eq:mixed}. Then,
  \begin{align*}
    \norm{g}{}^2 &= \ip{\Pi^0(u-u_h)}{\Pi^0(u-u_h)} = \ip{u-u_h}{\Pi^0(u-u_h)} = b(\uu-\uu_h,\vv).
  \end{align*}
  We proceed as in the proof of Theorem~\ref{thm:an}. 
  With $\ww = (g,0,\gamma_{0,\cS}v,\gamma_{\normal,\cS}\ttau)$ from Lemma~\ref{lem:w} and
  the bilinear form $a(\cdot,\cdot)$, we get, using
  Galerkin orthogonality, $\eeps=0$, and the bound $\norm{\eeps_r}V\lesssim \norm{\uu-\uu_h}U$
  (see~\cite[Theorem~2.1]{DPGaposteriori}), the estimate
  \begin{align*}
    b(\uu-\uu_h,\vv) &= a( (\uu-\uu_h,\eeps-\eeps_r),(\ww,\vv)) = a( (\uu-\uu_h,\eeps-\eeps_r),(\ww-\ww_h,\vv-\vv_r))
    \\
    &\lesssim \norm{\uu-\uu_h}U \big( \norm{\ww-\ww_h}U + \norm{\vv-\vv_r}V \big)
    \quad\text{for }(\ww_h,\vv_r) \in U_h\times V_r.
  \end{align*}
  We choose $\ww_h = (g,0,\gamma_{0,\cS}\Pi_\nabla^1 v,\gamma_{\normal,\cS}\Pi_\div^0\ttau)\in U_h = U_{hp}$ and
  $\vv_r = (\Pi_\nabla^1 v,\Pi_\div^0\ttau)\in V_r$. 
  Observe that $\ww-\ww_h = (0,0,\gamma_{0,\cS}(v-\Pi_\nabla^1 v),\gamma_{\normal,\cS}(\ttau-\Pi_\div^0\ttau))$. 
  With the latter choices we follow the same lines as in the proof of Theorem~\ref{thm:an} to see that
  \begin{align*}
    \norm{\ww-\ww_h}U + \norm{\vv-\vv_r}V \lesssim h^{s'}\norm{g}{}.
  \end{align*}
  Therefore, the last estimates show
  \begin{align*}
    \norm{g}{}^2 \lesssim h^{s'}\norm{\uu-\uu_h}U\norm{g}{}, \quad\text{hence},\quad
    \norm{g}{} \lesssim h^{s'}\norm{\uu-\uu_h}U.
  \end{align*}
  Applying Corollary~\ref{cor:apriori} to estimate $\norm{\uu-\uu_h}U$ shows~\eqref{eq:postproc:convRate}.
  If $\Omega$ is convex, then $s'=1$ and~\eqref{eq:postproc:convex} follows.
\end{proof}

\begin{figure}
  \begin{center}
    \includegraphics{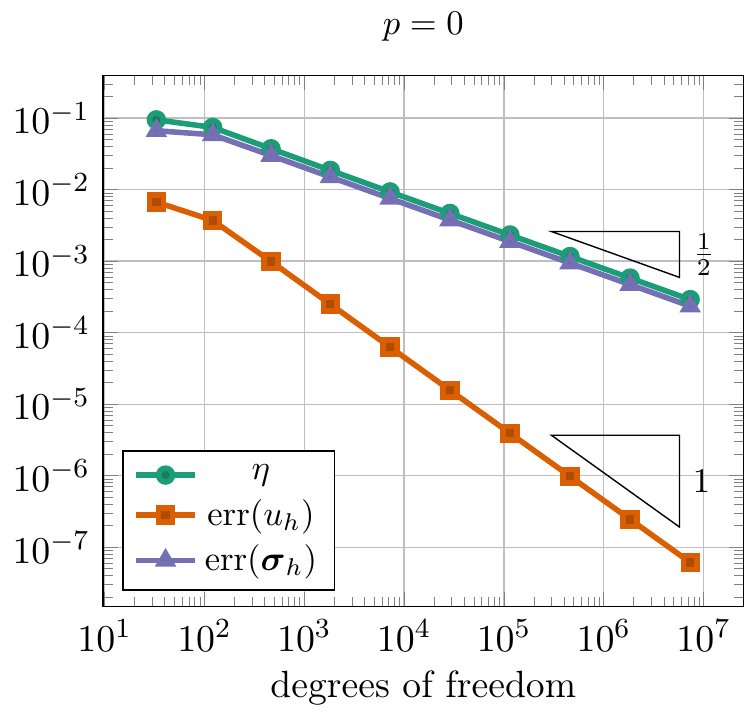}
    \includegraphics{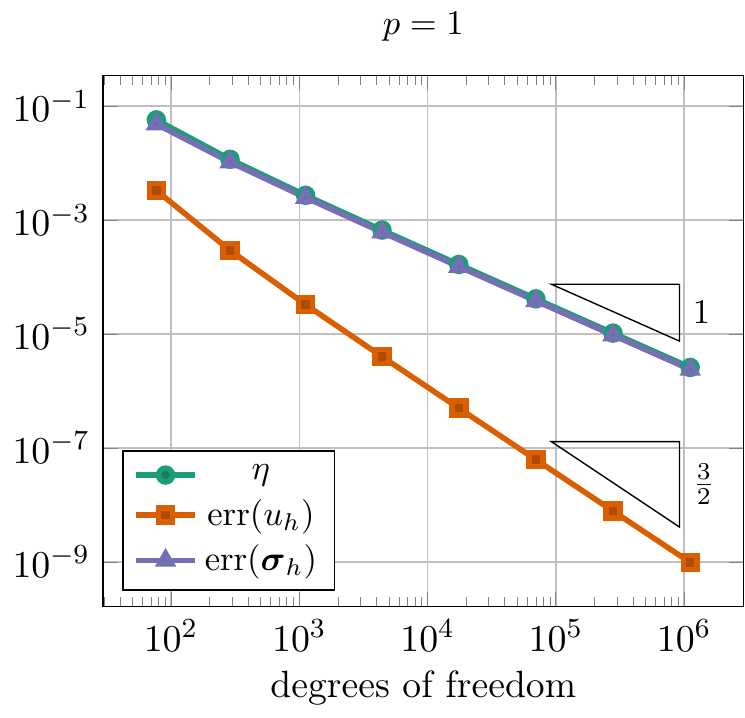}
    \includegraphics{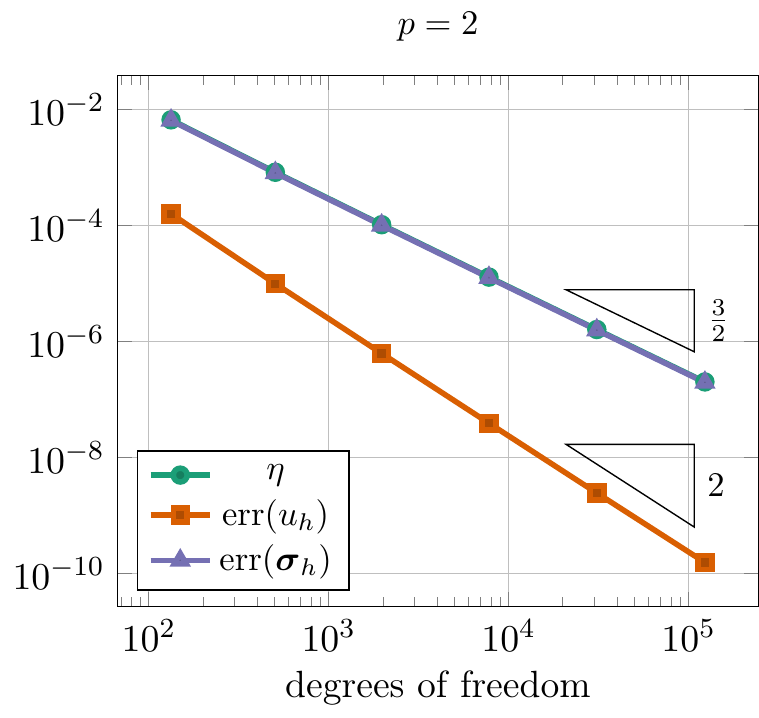}
    \includegraphics{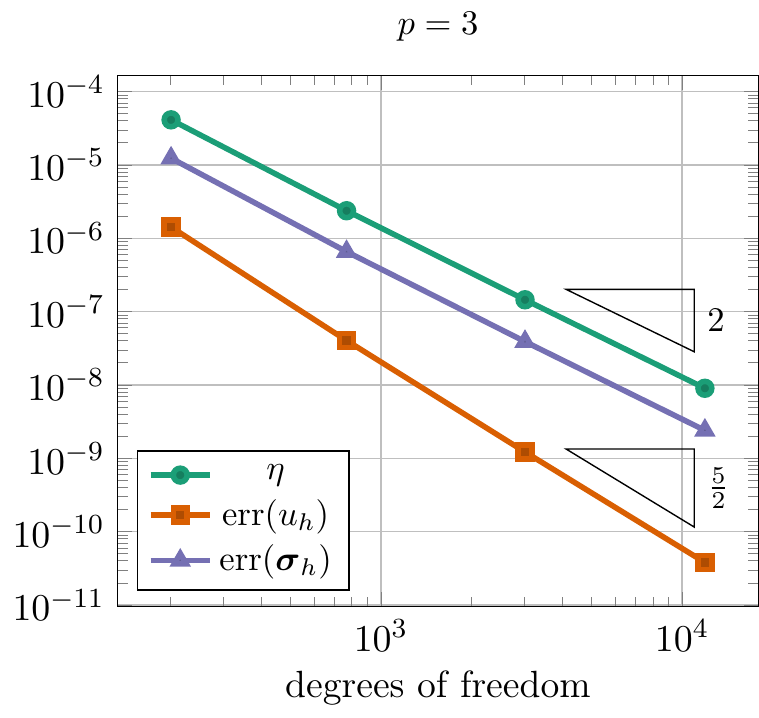}
  \end{center}
  \caption{Error plots for the solutions $\uu_h \in U_{hp}^+$ ($p=0,1,2,3$) in the augmented
  trial spaces (Section~\ref{sec:an}) for the example from Section~\ref{sec:ex:convex}.}
  \label{fig:convex:an}
\end{figure}

\begin{figure}
  \begin{center}
    \includegraphics{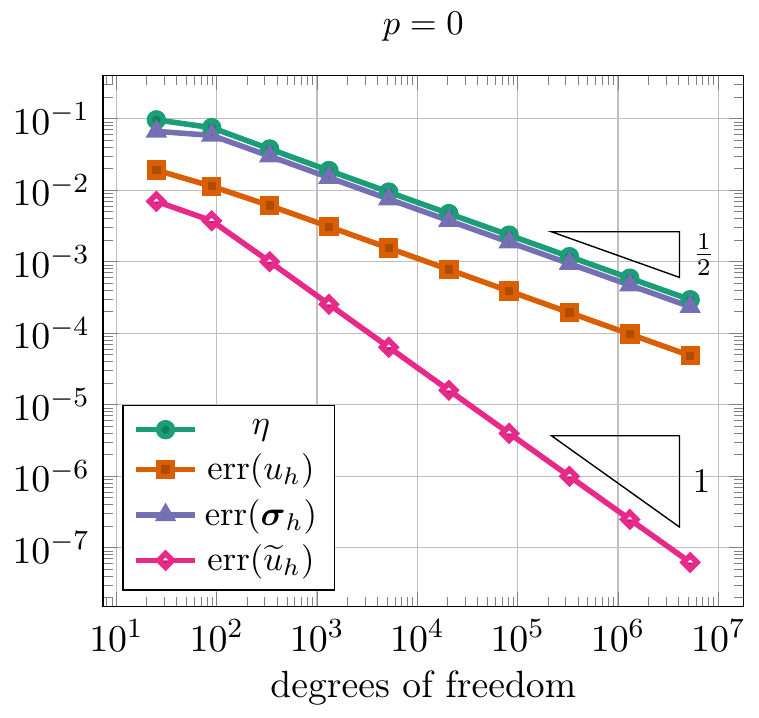}
    \includegraphics{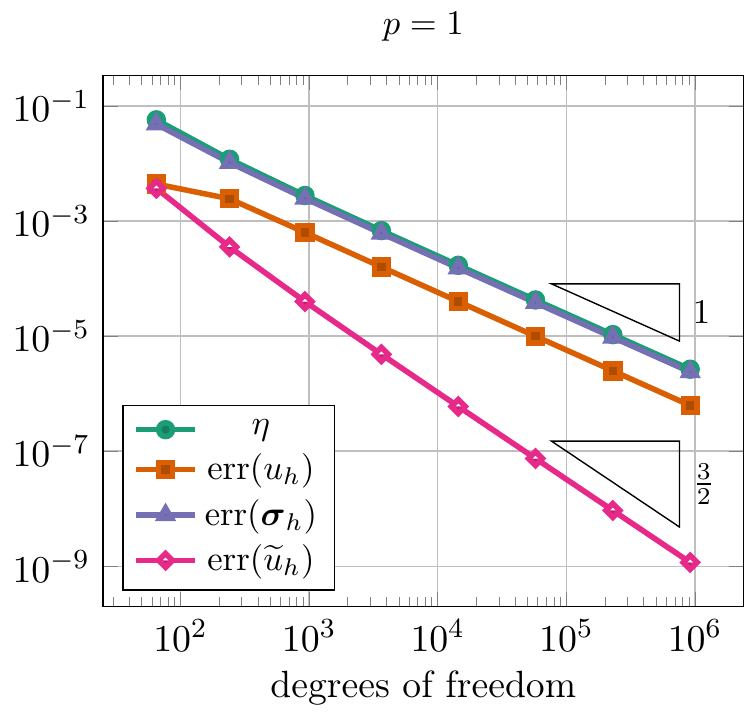}
    \includegraphics{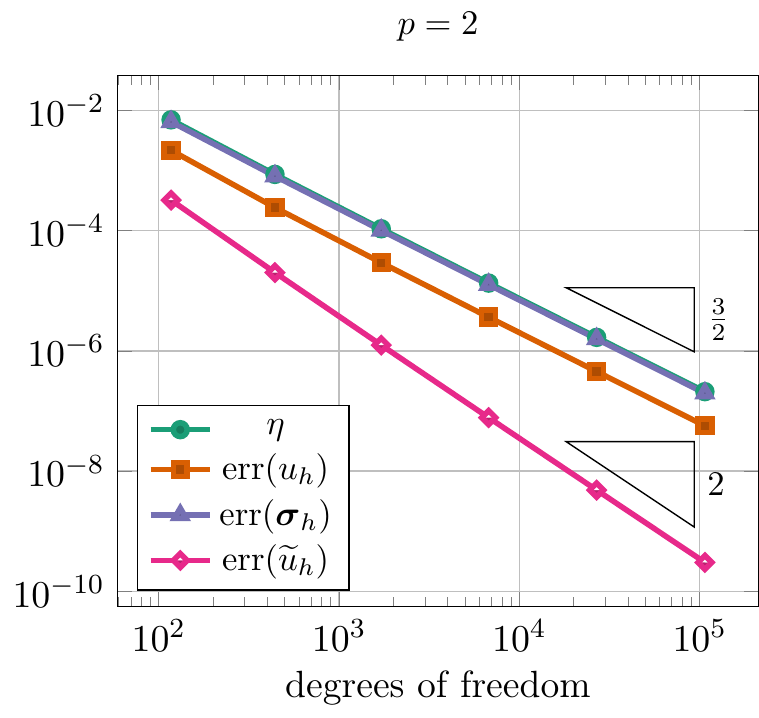}
    \includegraphics{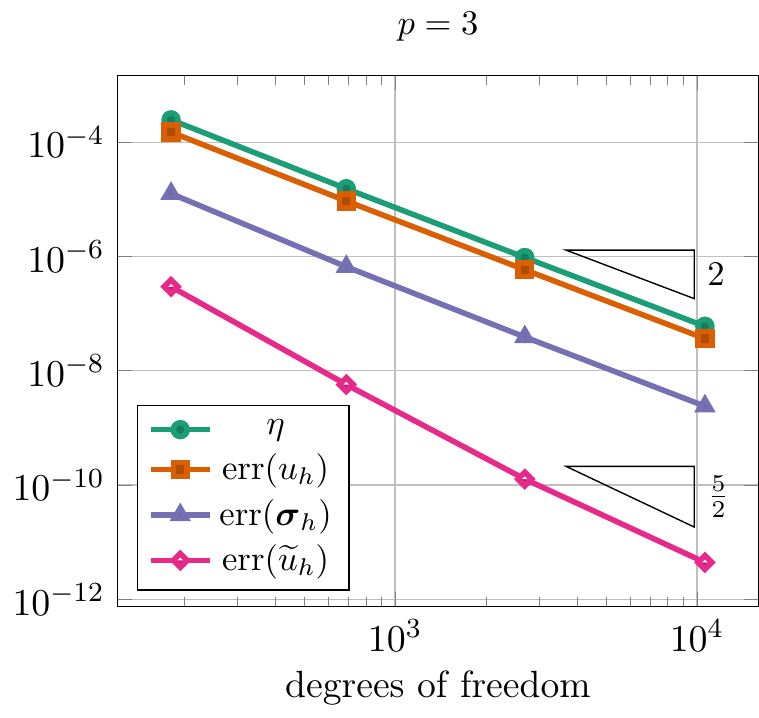}
  \end{center}
  \caption{Error plots for the postprocessed solution (Section~\ref{sec:postproc}) for $p=0,1,2,3$
  for the example from Section~\ref{sec:ex:convex}.}
  \label{fig:convex:postproc}
\end{figure}

\begin{figure}
  \begin{center}
    \includegraphics{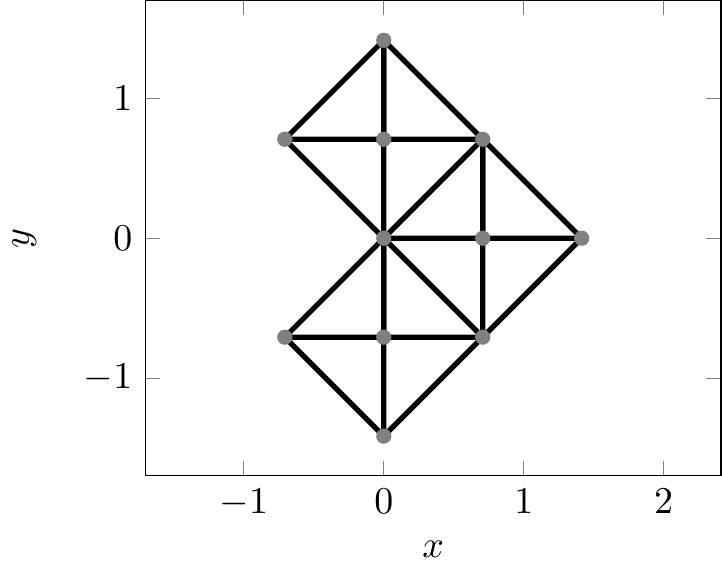}
  \end{center}
  \caption{L-shaped domain with initial triangulation (Section~\ref{sec:ex:nonconvex}).}
  \label{fig:Lshape}
\end{figure}

\begin{figure}
  \begin{center}
    \includegraphics{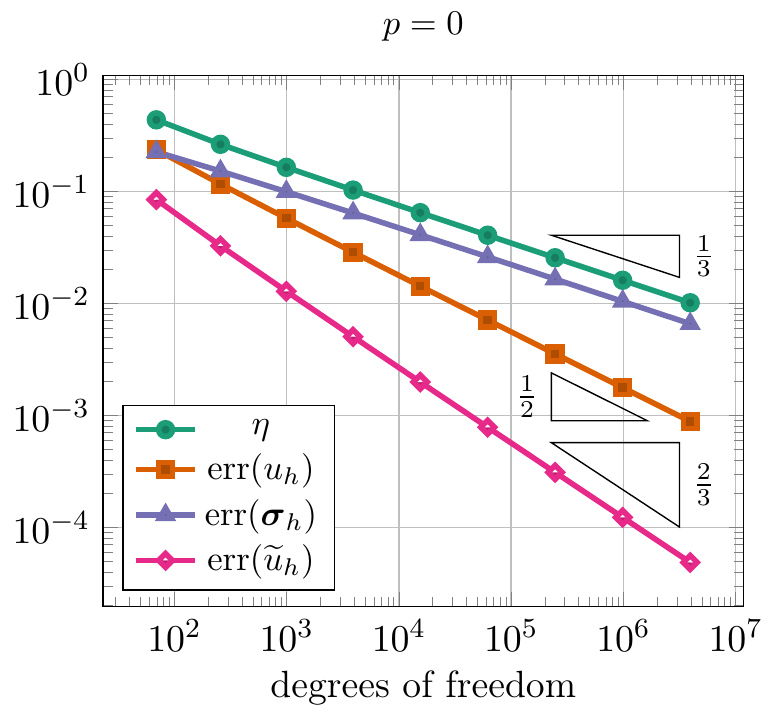}
    \includegraphics{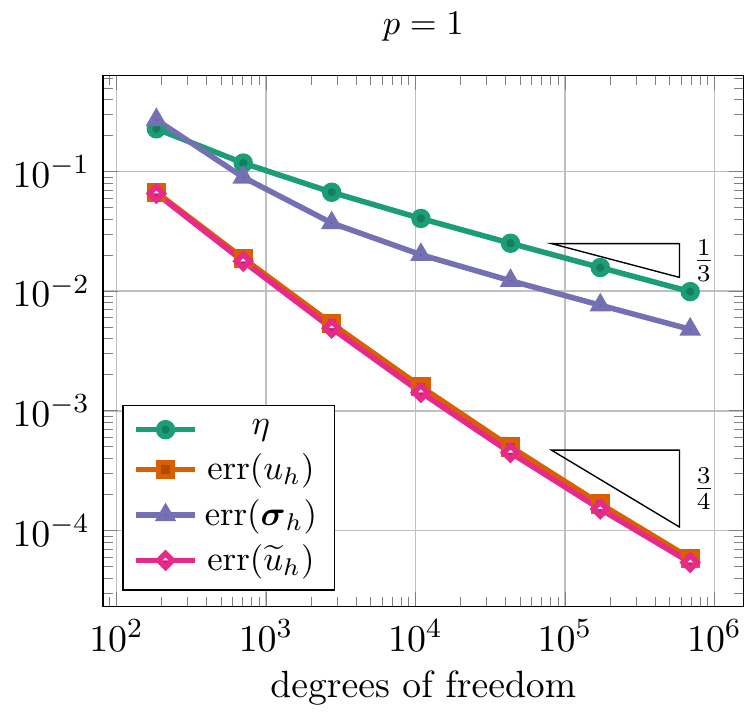}
    \includegraphics{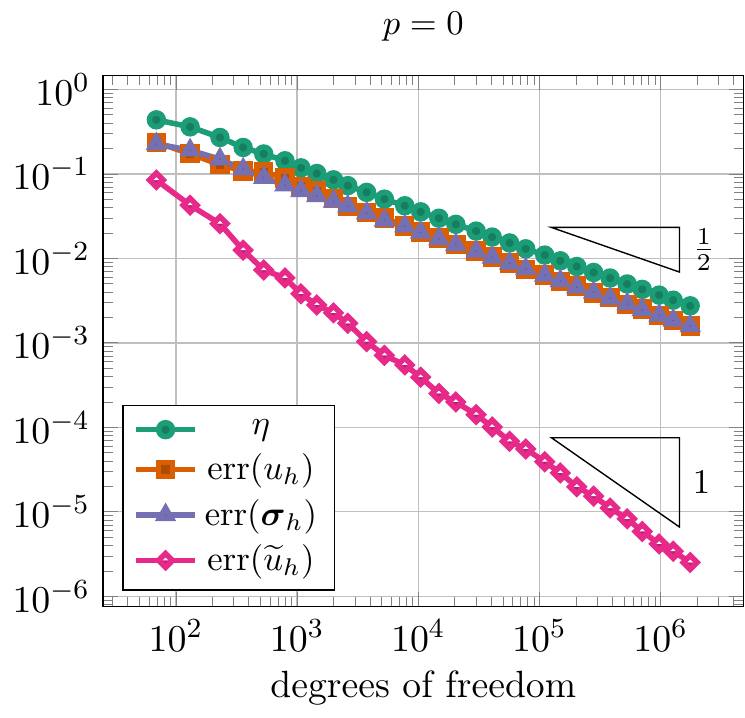}
    \includegraphics{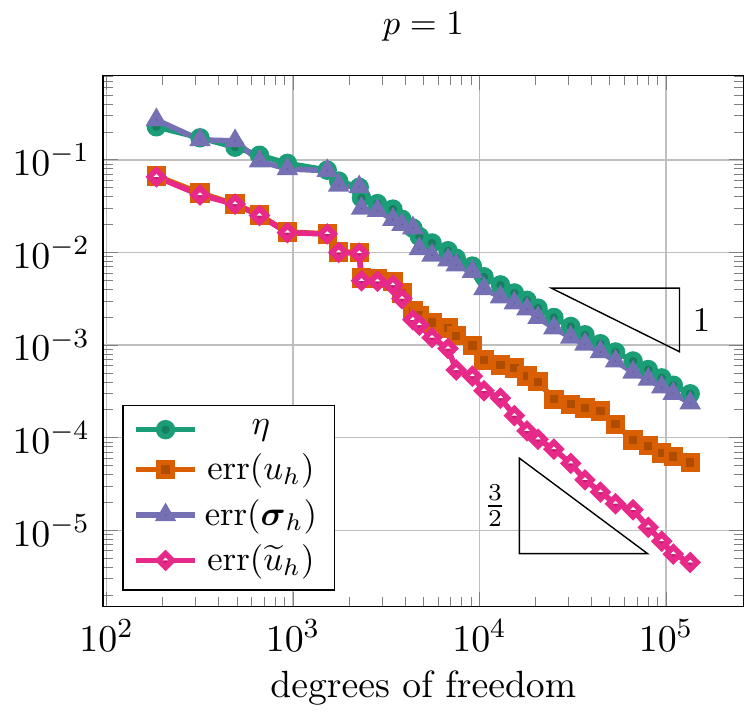}
  \end{center}
  \caption{Error plots for the postprocessed solution (Section~\ref{sec:postproc}) with $p=0$ (left) and $p=1$ (right)
    for the example from Section~\ref{sec:ex:nonconvex}.
    The first row corresponds to uniform refinements, the second to adaptive refinements.}
  \label{fig:nonconvex:postproc}
\end{figure}

\section{Numerical experiments}\label{sec:ex}
In this section we present some numerical examples with a convex domain (Section~\ref{sec:ex:convex}) 
for $p=0,1,2,3$ and nonconvex domain (Section~\ref{sec:ex:nonconvex}) for $p=0,1$.
In Section~\ref{sec:ex:convex} we consider the reaction-diffusion problem~\eqref{eq:model} and in
Section~\ref{sec:ex:nonconvex} we consider the Poisson equation.

Throughout, for $p\in\N_0$, we consider an enrichment $\Delta p=2$ for the test space, i.e., we use 
\begin{align*}
  V_r := \PP^{p+2}(\TT) \times \PP^{p+2}(\TT)^2.
\end{align*}
Let us note that by~\cite{practicalDPG} we would have to use a higher enrichment $\Delta p$ for the reaction-diffusion
problem to ensure the existence of a Fortin operator, in particular if we seek the solution $\uu_h$ in the augmented test
space $U_{hp}^+$, see also Remark~\ref{rem:testSpace}.
However, our experiments indicate that the enrichment $\Delta p=2$ is stable, meaning that
there is no significant difference in the results (concerning convergence rates) if $\Delta p >2$.
We refer to~\cite{BoumaGH_DPGconvRates}, where the topic of reduced test spaces was addressed for a primal DPG method.

In all the experiments we plot error quantities with respect to the degrees of freedom $D_h := \dim(U_h)$ in the trial
space.
Furthermore, experimental convergence rates are indicated with triangles in the plots. 
The number $\alpha$ next to the triangle is its negative slope.
Note that for uniform refinements it holds that $D_h^{-1} \simeq h^2$. 
Thus, a slope $\alpha$ corresponds to a straight line parallel to $D_h^{-\alpha} \simeq
h^{2\alpha}$.
Let $\uu=(u,\ssigma,\widehat u,\widehat\sigma)\in U$ be the exact solution of~\eqref{eq:uwf}.
We define the error quantities
\begin{align*}
  \err(v) := \norm{u-v}{}, \quad \err(\ttau) := \norm{\ssigma-\ttau}{} \quad\text{for }v\in L^2(\Omega), \ttau\in
  L^2(\Omega)^2.
\end{align*}
Moreover, for $(\uu_h,\eeps_r)\in U_h\times V_r$ being the solution of the first-order system~\eqref{eq:mixedDPG}, 
let $\eta := \eta(\uu_h) := \norm{\eeps_r}V$ denote the DPG error estimator. We refer the interested reader
to~\cite{DPGaposteriori} for more information on adaptivity for DPG methods.
We will use this localizable estimator to steer an adaptive algorithm in Section~\ref{sec:ex:nonconvex}.

\subsection{Convex domain}\label{sec:ex:convex}
Let $\Omega = (0,1)^2$. We prescribe the exact solution
\begin{align*}
  u(x,y) = x(1-x)y(1-y) \in H_0^1(\Omega).
\end{align*}
Note that $u$ is smooth, thus, $u\in H^{1+s}(\Omega)$ for all $s>1/2$. Moreover, since $\Omega$ is convex, the
regularity shift~\eqref{eq:mixed:regshift} satisfies $s'=1$.
Therefore, we expect a convergence behavior $\err(u_h) = \OO(h^{p+2})$, if $\uu_h\in U_{hp}^+$ is the solution
in the augmented trial space as well as $\err(\widetilde u_h) = \OO(h^{p+2})$ for the postprocessed
solution~\eqref{eq:postproc}, where $\uu_h\in U_{hp}$.

Figure~\ref{fig:convex:an} shows convergence plots for solutions in the augmented space $U_{hp}^+$ for $p=0,1,2,3$.
We observe the optimal rates as predicted by Corollary~\ref{cor:an} (convex domain).

Figure~\ref{fig:convex:postproc} shows convergence plots for the postprocessed solution $\widetilde\uu_h$
(Section~\ref{sec:postproc}) for $p=0,1,2,3$.
Again we observe optimal convergence as predicted by Theorem~\ref{thm:postproc}.

\subsection{Nonconvex domain}\label{sec:ex:nonconvex}
We consider the L-shaped domain visualized in Figure~\ref{fig:Lshape} and the Poisson equation
\begin{alignat*}{2}
  -\Delta u &= f &\quad&\text{in }\Omega, \\
  u &= u_\Gamma &\quad&\text{on }\Gamma.
\end{alignat*}
The ultra-weak bilinear form is given by (see, e.g.,~\cite{DemkowiczG_11_ADM})
\begin{align*}
  b(\uu,\vv) = \ip{u}{\pwdiv\ttau} + \ip{\ssigma}{\pwnabla v + \ttau} -\dual{\widehat u}{\ttau\cdot\normal}_\cS 
  -\dual{\widehat\sigma}v_\cS.
\end{align*}
We incorporate the inhomogeneous Dirichlet data as described in~\cite{DPGoverview}.
In this section we only discuss results for the postprocessed solutions. The results for solutions in the augmented space
$U_{hp}^+$ are similar.

We prescribe the solution
\begin{align*}
  u(r,\varphi) = r^{2/3}\cos(2/3\varphi),
\end{align*}
where $(r,\varphi)$ denote the polar coordinates of $(x,y)\in\Omega$. Then, $f:=-\Delta u = 0$.
It is a straightforward calculation to check that $u\in H^{1+2/3-\eps}(\Omega)$ for all $\eps>0$.
In view of the a priori result from Corollary~\ref{cor:apriori} 
we expect a (numerical) convergence behavior $\norm{\uu-\uu_h}U = \OO(h^{2/3}) = \OO(D_h^{-1/3})$ for the overall error
and uniform mesh-refinement ($p=0$ and $p=1$).
This perfectly fits to the plots in the first row of Figure~\ref{fig:nonconvex:postproc}.
We note that $\err(u_h) = \norm{u-u_h}{}$ has a higher convergence rate than the overall error,
i.e., this error behaves like $h^1$ for $p=0$ and $h^{3/2}$ for $p=1$.
If we take a look at the convergence order of $\err(\widetilde u_h)$ we see that these rates are far from optimal. In
particular, for $p=0$ we get $\err(\widetilde u_h) = \OO(h^{4/3}) = \OO(D_h^{-2/3})$. 
From Theorem~\ref{thm:postproc} we learned that the error of the postprocessed solution depends 
on the regularity of the solution and the regularity of some dual problem.
For the particular domain configuration we stress that one can not expect higher regularity than
$H^{1+2/3-\eps}(\Omega)$ for the dual problem and the problem itself.
This is also observed in the upper left plot in Figure~\ref{fig:nonconvex:postproc}. 
Moreover, this also shows that the bounds in Theorem~\ref{thm:postproc} can, in general, not be sharpened.
For $p=1$ (upper right plot in Figure~\ref{fig:nonconvex:postproc}) we get a slightly better behavior of $h^{3/2}$
compared to $h^{4/3}$.

Finally, since adaptivity is quite a ``natural'' concept in the context of DPG methods, see~\cite{DPGaposteriori},
we consider also an adaptive algorithm (with local mesh refining)
of the type
\begin{align*}
  \boxed{\texttt{SOLVE}} \to \boxed{\texttt{ESTIMATE}} \to \boxed{\texttt{MARK}}
  \to\boxed{\texttt{REFINE}}.
\end{align*}
The marking step is realized using the bulk criterion: Find a set (of minimal cardinality) $\MM\subseteq \TT$ such that
\begin{align*}
  \theta \eta^2 \leq \sum_{T\in\MM} \eta(T)^2,
\end{align*}
where $\theta\in(0,1)$, and $\eta(T)$ are the local contributions of the DPG error estimator $\eta$, i.e.,
\begin{align*}
  \eta^2 = \sum_{T\in\TT}\eta(T)^2.
\end{align*}
Throughout, we choose $\theta=\tfrac14$.
As refinement rule we use newest-vertex bisection, 
see, e.g.,~\cite{stevenson:NVB} for details on the refinement strategy.

For this problem we observe that the optimal rates (with respect to $D_h$) for the overall error
and for the postprocessed solution are recovered, i.e., $\norm{\uu-\uu_h}U = \OO(D_h^{-(p+1)/2})$ and 
$\norm{u-\widetilde u_h}{} = \OO(D_h^{-(p+2)/2})$, see 
second row of plots in Figure~\ref{fig:nonconvex:postproc}.

\bibliographystyle{abbrv}
\bibliography{literature}

\end{document}